\pdfoutput=1

\documentclass[a4paper,12pt]{amsart}

\usepackage[utf8]{inputenc}
\usepackage{amsfonts}
\usepackage{amssymb}
\usepackage{amsmath}
\usepackage{amsthm}
\usepackage{graphicx}
\usepackage{float}
\usepackage{wrapfig}
\usepackage{tikz}
\usepackage{standalone}
\usepackage{enumerate}
\usepackage{enumitem}
\usepackage{tkz-euclide}

\usepackage{url}
\usepackage{hyperref}

\newtheorem{theorem}{Theorem}[section]
\newtheorem{cor}{Corollary}[section]

\newtheorem{lemma}{Lemma}[section]
\newtheorem{prop}{Proposition}[section]

\theoremstyle{definition}

\newtheorem{rem}{Remark}[section]

\usepackage[backend=biber]{biblatex}

\DeclareFieldFormat[article,book,inbook,incollection,inproceedings,patent,thesis,unpublished]{title}{\textit{#1}}

\begin{document}

\title{Isogonal conjugation in isosceles tetrahedron}
\author{Saro Harutyunyan}
\address{Saro Harutyunyan, Yerevan, Armenia 0012}
\email{saro.harutyunyan@gmail.com}


\begin{abstract}
In this article we investigate the properties of isogonal conjugation in isosceles tetrahedron.
Particularly we reveal three hyperbolic paraboloids each of which is formed by pairs of isogonal conjugate points symmetric in the respective bimedian, as well as we prove that the circumsphere of an isosceles tetrahedron is invariant under isogonal conjugation in that tetrahedron.
\end{abstract}

\maketitle

\section{Introduction}

Tetrahedron \(ABCD\) is called \textit{isosceles} (or \textit{equihedral}) if its opposite edges are equal, i.e. \(AB=CD,\ BC=AD,\ AC=BD\). This type of tetrahedron is already well-investigated. One may refer to \cite{W} for a list of its known properties. For the purposes of this paper we will need only a few of these properties which will be discussed in Section~\ref{sect:iso-tetr}.

One may note that the isosceles tetrahedron is kind of a generalization of the equilateral triangle and is inherently ``symmetric" so it should have a ``center" that will coincide with its circumcenter, incenter and centroid. Actually this is true for the high-dimensional analogue of isosceles tetrahedron too, see \cite{E05} and~\cite{EMH05}.

We will need the notion of \emph{isogonal conjugation} with respect to a polyhedron. This is the natural generalization of this transformation for polygons. First, for a given dihedron \(\mathcal D\) with edgeline \(e\) and a point \(P\) define the \emph{isogonal plane} of \(Pe\) in \(\mathcal D\) as the plane symmetric to \(Pe\) with respect to the bisector plane of \(\mathcal D\) (if \(P\) lies on \(e\) then the plane \(Pe\), as well as its isogonal can be any plane through \(e\)). Then, for a given polyhedron \(\mathcal P\) and a point \(P\) define the \emph{isogonal conjugate} of \(P\) in \(\mathcal P\) as the point \(Q\) (in case of existence) so that \(P\) and \(Q\) lie in isogonal planes in each dihedron of \(\mathcal P\).
Obviously if \(P\) is the isogonal conjugate of \(Q\), then \(Q\) is the isogonal conjugate of \(P\).

Isogonal conjugation is well-defined for an arbitrary tetrahedron, that is, any point of space has an isogonal conjugate with respect to tetrahedron. On the other hand this is not the case with other polyhedra, there might be only a few or not even a single point which have isogonal conjugates. Anyway we will only work with tetrahedron and will give a proof of the first sentence of this paragraph in Section~\ref{sect:isogonality}.

Section~\ref{sect:aux-facts} will address several auxiliary facts, mainly concerning circles and spheres, which will be leveraged later.

One of the main results of this paper will be presented in Section~\ref{sect:iso-sym-bimed} where we consider isogonal conjugate pairs symmetric in bimedians of isosceles tetrahedron. We prove that they lie on three hyperbolic paraboloids. This is the only section where we will utilize coordinate bashing opposed to the geometric techniques used elsewhere.

The other major result of this paper is that the circumsphere (except for the vertices) of isosceles tetrahedron is invariant with respect to isogonal conjugation. In other words, the isogonal conjugate of any point of the circumsphere of an isosceles tetrahedron other than a vertex lies on its circumsphere. This will be proved in Section~\ref{sect:iso-conj-sphere}.

Interestingly, the two-dimensional case of isosceles tetrahedron,\break
namely the equilateral triangle, does not have any similar property. Moreover, the isogonal conjugate of any point \(P\notin\{A, B, C\}\) of the circumcircle of \emph{any} triangle \(ABC\) is ``infinite", i.e. the isogonals of \(AP, BP, CP\) are parallel.

\section{Properties of isosceles tetrahedron}
\label{sect:iso-tetr}

\begin{prop}
    Isosceles tetrahedron has the following properties:
    \begin{enumerate}[label=(\roman*)]
        \item \label{i} any of its edges is seen in equal angles from the other two vertices and its faces are congruent triangles (hence the name equihedral);
        \item \label{ii} its circumcenter and incenter coincide;
        \item \label{iii} its faces are acute triangles.
    \end{enumerate}
\end{prop}

\begin{proof}
\begin{enumerate}[label=(\roman*)]
    \item Obvious by the equality of opposite edges.
    \item Let \(O\) be the center of circumsphere \(\Omega\) of isosceles tetrahedron \(ABCD\). The locus of points \(Z\in \Omega\) such that \(\angle AZB=\angle ACB\) is a union of two circular arcs which are symmetric with respect to the plane \(ABO\). This means that the planes \(ADB\) and \(ACB\) are symmetric with respect to \(ABO\) for \(\angle ADB=\angle ACB\), which follows from~\ref{i}. So \(O\) lies on the bisector plane of dihedron \(AB\). Similarly \(O\) lies on the bisector planes of the other dihedrons and thus coincides with the incenter.
    \item Assume for the sake of contradiction that \(\angle ABC\ge 90^\circ\), for example. Choose the point \(D'\) in such a way that \(ABCD'\) is a parallelogram. Then if \(M\) is the midpoint of \(AC\), we have \(\triangle ADC=\triangle AD'C\) and
    \[BD<BM+MD=BM+MD'=BD'\le AC\]
    which is obviously false (the last inequality follows from the fact that \(BD'\) lies inside the circle with diameter \(AC\)). Thus the faces of \(ABCD\) are acute angled.
\end{enumerate}
\end{proof}

\section{Isogonal conjugation in tetrahedron}
\label{sect:isogonality}

Here we will work out a proof of the correctness of isogonal conjugation in tetrahedron. Note that by definition if \(P\) is a vertex of the tetrahedron then for any \(Q\) lying in the opposite faceplane \(P\) and \(Q\) are isogonal conjugates. Note as well that any two points on the opposite edgelines of the tetrahedron are isogonal conjugates too. Thus we may proceed assuming that \(P\) does not lie on the surface of the tetrahedron.

\begin{theorem}
\label{theorem:iso-conj}
    For arbitrary tetrahedron \(ABCD\) and point \(P\) not lying on its surface exists its isogonal conjugate \(Q\) with respect to \(ABCD\).
\end{theorem}

\begin{proof}
    Let \(P_A, P_B, P_C, P_D\) be the reflections of \(P\) in the respective faces of \(ABCD\). Since \(DP_A=DP_B=DP_C\), the line through \(D\) perpendicular to \(P_AP_BP_C\) passes through the circumcenter \(Q\) of \(P_AP_BP_CP_D\). The lines through \(A, B, C\) defined similarly pass through \(Q\) too.

    Now let us show that, for example, the planes \(ABP\) and \(ABQ\) are isogonal in dihedron \(AB\), i.e. they make equal angles with its bisector plane; see Figure~\ref{fig:isogon-conj} for a perspective in the direction of edge \(AB\). Choose a positive direction of rotation around \(AB\) and define \(\angle(ABX, ABY)\) for \(X\) and \(Y\) not lying on \(AB\) as the minimal angle of rotation in that positive direction that sends the plane \(ABX\) to \(ABY\). Then
    \begin{align*}
        \angle(ABP, ABD)&=\frac{\angle(ABP, ABP_C)}{2}\\
        &=\frac{\angle(ABP_D, ABP_C)-\angle(ABP_D, ABP)}{2}\\
        &=\angle(ABP_D, ABQ)-\angle(ABP_D, ABC)\\
        &=\angle(ABC, ABQ).
    \end{align*}
    
    Similarly planes \(eP\) and \(eQ\) are isogonal in dihedron \(e\) for any other edge \(e\) of \(ABCD\). Hence \(P\) and \(Q\) are isogonal conjugates.
    
    \begin{figure}[H]
    \definecolor{uuuuuu}{rgb}{0.26666666666666666,0.26666666666666666,0.26666666666666666}

\begin{tikzpicture}[line cap=round,line join=round,x=1cm,y=1cm]

\clip(-7,-1) rectangle (2.5,3);

\coordinate (AB) at (-4.28,-0.12);
\coordinate (C) at (-2.74,2.32);
\coordinate (D) at (-0.32,-0.16);
\coordinate (P) at (-2.02,0.44);
\coordinate (Q) at (-2.743614708630837,1.3251063127197225);
\coordinate (P_C) at (-2.0317731075290757,-0.7255376453784943);
\coordinate (P_D) at (-4.7468726276846205,2.161058953538654);

\draw [shift={(AB)},line width=0.5pt] (0,0) -- (-0.5787255656077621:0.47810941875554924) arc (-0.5787255656077621:13.916877433301273:0.47810941875554924) -- cycle;
\draw [shift={(AB)},line width=0.5pt] (0,0) -- (-15.074328564516799:0.5464107642920563) arc (-15.074328564516799:-0.5787255656077607:0.5464107642920563) -- cycle;
\draw [shift={(AB)},line width=0.5pt] (0,0) -- (43.24642565927582:0.6147121098285633) arc (43.24642565927582:57.742028658184864:0.6147121098285633) -- cycle;
\draw [shift={(AB)},line width=0.5pt] (0,0) -- (13.916877433301273:0.40980807321904217) arc (13.916877433301273:57.742028658184864:0.40980807321904217) -- cycle;
\draw [shift={(AB)},line width=0.5pt] (0,0) -- (57.742028658184864:0.34150672768253515) arc (57.742028658184864:101.56717988306845:0.34150672768253515) -- cycle;
\draw[line width=0.5pt] (-2.022959651889765,0.14699446291327245) -- (-2.3127229374922846,0.14992136478804535) -- (-2.3156498393670573,-0.13984192081447425) -- (-2.0258865537645376,-0.14276882268924712) -- cycle; 
\draw[line width=0.5pt] (-3.2200684082198,1.1974202248600379) -- (-3.116959156310511,1.3607881304825478) -- (-3.2803270619330207,1.463897382391837) -- (-3.38343631384231,1.300529476769327) -- cycle; 

\draw [line width=0.5pt] (AB)-- (D);
\draw [line width=0.5pt] (AB)-- (C);
\draw [line width=0.5pt] (C)-- (D);
\draw [line width=0.5pt] (P_D)-- (AB);
\draw [line width=0.5pt] (AB)-- (P_C);
\draw [line width=0.5pt] (P)-- (AB);
\draw [line width=0.5pt] (AB)-- (Q);
\draw [line width=0.5pt,dash pattern=on 1pt off 1pt] (P)-- (P_C);
\draw [line width=0.5pt,dash pattern=on 1pt off 1pt] (P)-- (P_D);
\draw [shift={(AB)},line width=0.5pt] (-0.5787255656077621:0.47810941875554924) arc (-0.5787255656077621:13.916877433301273:0.47810941875554924);
\draw[line width=0.5pt] (-3.849221189296669,-0.06963081662272741) -- (-3.761030251672365,-0.059319015301396015);
\draw [shift={(AB)},line width=0.5pt] (-15.074328564516799:0.5464107642920563) arc (-15.074328564516799:-0.5787255656077607:0.5464107642920563);

\draw[line width=0.5pt] (-3.782661429173349,-0.18836150555617762) -- (-3.6946967839931255,-0.20045265619876695);
\draw [shift={(AB)},line width=0.5pt] (43.24642565927582:0.6147121098285633) arc (43.24642565927582:57.742028658184864:0.6147121098285633);
\draw[line width=0.5pt] (-3.9171899266468406,0.32003347468142174) -- (-3.860704526005031,0.3885416803204458);

\draw [shift={(AB)},line width=0.5pt] (13.916877433301273:0.40980807321904217) arc (13.916877433301273:57.742028658184864:0.40980807321904217);
\draw [shift={(AB)},line width=0.5pt] (13.916877433301273:0.3346765931288844) arc (13.916877433301273:57.742028658184864:0.3346765931288844);
\draw [shift={(AB)},line width=0.5pt] (57.742028658184864:0.34150672768253515) arc (57.742028658184864:101.56717988306845:0.34150672768253515);
\draw [shift={(AB)},line width=0.5pt] (57.742028658184864:0.2663752475923774) arc (57.742028658184864:101.56717988306845:0.2663752475923774);

\begin{scriptsize}
\draw [fill=black] (AB) circle (1pt);
\draw[color=black] (-4.471911754661568,-0.34773990847619896) node {$AB$};
\draw [fill=black] (D) circle (1pt);
\draw[color=black] (-0.1,-0.2) node {$D$};
\draw [fill=black] (C) circle (1pt);
\draw[color=black] (-2.7,2.55) node {$C$};
\draw [fill=black] (P) circle (1pt);
\draw[color=black] (-1.9,0.65) node {$P$};
\draw [fill=black] (P_C) circle (1pt);
\draw[color=black] (-1.753518202308588,-0.7438877125879397) node {$P_C$};
\draw [fill=black] (P_D) circle (1pt);
\draw[color=black] (-4.8,2.4) node {$P_D$};
\draw [fill=black] (Q) circle (1pt);
\draw[color=black] (-2.6,1.5) node {$Q$};
\draw [fill=uuuuuu] (-3.38343631384231,1.300529476769327) circle (1pt);
\draw [fill=uuuuuu] (-2.0258865537645376,-0.14276882268924712) circle (1pt);
\end{scriptsize}

\end{tikzpicture}
        \caption{}
        \label{fig:isogon-conj}
    \end{figure}
\end{proof}

Recall that for a given tetrahedron \(ABCD\) and point \(P\) the sphere passing through the projections of \(P\) onto the faces of \(ABCD\) is called the \emph{pedal sphere} of \(P\). For degenerated cases, i.e. when the set of projections contains less than four points, the pedal sphere can be defined via limit.

\begin{rem}
\label{rem:iso-conj}
    Note that in the proof of Theorem~\ref{theorem:iso-conj} the homothety with coefficient \(1/2\) centered at \(P\) sends the vertices of \(P_AP_BP_CP_D\) and its circumcenter \(Q\) to the projections of \(P\) on the faces of \(ABCD\) and the midpoint \(M_{PQ}\) of \(PQ\), respectively. So the resulting sphere centered at \(M_{PQ}\) passes through the projections of \(P\). Similarly it passes through the projections of \(Q\) too. Hence we may formulate the following proposition which is the generalization of the respective result for the triangle:
    \begin{cor}
    \label{cor:iso-conj-pedal-sphere}
        The eight projections of two isogonal conjugate points in a tetrahedron onto its faces lie on a (pedal) sphere centered at the midpoint of the segment joining the two isogonal conjugate points. Moreover, the projections on the same face are diametrically opposite in the intersection circle of the sphere and the face.
    \end{cor}
\end{rem}

It is not difficult to see that the reasoning in the proof of Theorem~\ref{theorem:iso-conj} and Remark~\ref{rem:iso-conj} are reversible. This lets us formulate the following corollary:
\begin{cor}
\label{cor:equiv-iso-conj-constr}
    Points \(P\) and \(Q\) are isogonal conjugates in a tetrahedron iff their pedal spheres coincide.
\end{cor}

\section{Auxiliary facts}
\label{sect:aux-facts}

When working with spheres it is useful to study the analogous configurations (if existing) for circles on the plane. Many properties of circles on the plane are valid for spheres too. Particularly, this concerns to inversion.

Define the \emph{angle between two intersecting spheres} as the angle between their tangent planes in a point of their intersection. This definition can be extended for a sphere and a plane too. Indeed, one may think of the plane as a special case of a sphere whose center is at infinity.

It is well-known that angles between circles are preserved under inversion. Naturally, this is the case with spheres too.

\begin{prop}
\label{prop:angles-preserved}
    Angles between spheres are preserved under inversion.
\end{prop}
\begin{proof}
    Let us be given two intersecting spheres \(\gamma_1, \gamma_2\) and an inversion sphere \(\Omega\). Consider their section with the plane \(\pi\) through their centers. Then the angle between \(\gamma_1\) and \(\gamma_2\) is equal to the angle between the circles \(\gamma_1\cap \pi\) and \(\gamma_2\cap \pi\). Also note that for these two circles inversion in \(\Omega\) is equivalent to inversion in \(\Omega\cap \pi\). Thus, since the angles between circles are preserved under inversion on plane, the angles between the spheres constructed on these circles having the same center and radius are also preserved.
\end{proof}

The second fact that will come handy for the proof of the main result is as well a generalization of a plane construction:

\begin{prop}
\label{prop:bisector-sphere}
    Let us be given a sphere \(\Omega\) and a circle \(\sigma\) on it, as well as a sphere \(\Gamma\) passing through \(\sigma\). Suppose that \(\Gamma\) makes equal angles with \(\Omega\) and the plane of \(\sigma\). Then the center of \(\Gamma\) lies on \(\Omega\).
\end{prop}
\begin{proof}
    Let \(Q\) be the center of \(\Gamma\).
    Consider a section of the construction with a plane through the centers of \(\Gamma\) and \(\Omega\). Let the sections of these spheres be the circles \(\gamma\) and \(\omega\) respectively, and let \(S_1, S_2\) be the points of intersection of \(\sigma\) with the secant plane; see Figure~\ref{fig:bisector-sphere}. Let also \(UV\) be the diameter in \(\omega\) perpendicular to \(S_1S_2\) and \(t\) be the tangent at \(S_2\) to \(\omega\). Without loss of generality we may assume that \(Q\) and \(V\) lie in the same side of \(S_1S_2\).

    \begin{figure}[H]
    \definecolor{uuuuuu}{rgb}{0.26666666666666666,0.26666666666666666,0.26666666666666666}

\begin{tikzpicture}[line cap=round,line join=round,x=1cm,y=1cm]

\clip(-6.5,0) rectangle (5.5,5.5);

\coordinate (S_1) at (0.14,0.98);
\coordinate (S_2) at (0.12776715089424484,3.875159308867408);
\coordinate (U) at (-3.965104149098117,2.410260297199328);
\coordinate (V) at (0.6451041490981171,2.4297397028006715);

\draw [shift={(S_2)},line width=0.5pt] (0,0) -- (129.14398641457106:0.5) arc (129.14398641457106:199.69303774740575:0.5) -- cycle;
\draw [shift={(S_2)},line width=0.5pt] (0,0) -- (-160.30696225259425:0.4) arc (-160.30696225259425:-89.75791091975952:0.4) -- cycle;
\draw [line width=0.5pt] (-1.66,2.42) circle (2.3051247254758254cm);
\draw [line width=0.5pt] (S_1)-- (S_2);
\draw [line width=0.5pt] (V) circle (1.5352118444412528cm);
\draw [line width=0.5pt] (U)-- (S_2);
\draw [line width=0.5pt] (0.5421272690799566,3.3660883065249623)-- (-1.1572001228265658,5.453833388010119);

\begin{scriptsize}
\draw [fill=uuuuuu] (S_1) circle (1pt);
\draw[color=uuuuuu] (0.3,0.7) node {$S_1$};
\draw[color=black] (-3.2,4.35) node {$\omega$};
\draw [fill=uuuuuu] (S_2) circle (1pt);
\draw[color=uuuuuu] (0.3,4.2) node {$S_2$};
\draw [fill=uuuuuu] (V) circle (1pt);
\draw[color=uuuuuu] (1.3,2.41) node {$V=Q$};
\draw[color=black] (1.9,3.6) node {$\gamma$};
\draw [fill=uuuuuu] (U) circle (1pt);
\draw[color=uuuuuu] (-4.26,2.43) node {$U$};
\draw[color=black] (-0.8,5.3) node {$t$};
\end{scriptsize}

\end{tikzpicture}
        \caption{}
        \label{fig:bisector-sphere}
    \end{figure}
    
    By simple angle chasing one may check that \(US_2\) bisects the angle between \(t\) and \(S_1S_2\). This means that \(\gamma\) should touch \(US_2\). Similarly \(\gamma\) should touch \(US_1\) too so the center of \(\gamma\) coincides with \(V\) and thus lies on \(\Omega\).
\end{proof}

We will utilize the following property of isogonal conjugate points in a triangle as well:

\begin{prop}
\label{prop:iso-conj-bisect-circ}
    Let \(X\) and \(Y\) be isogonal conjugate points in the triangle \(ABC\).
    Let \(M\) and \(N\) be the midpoints of arcs \(AB\) not containing and containing \(C\), respectively. Let \(\gamma_M\) and \(\gamma_N\) be the circles centered at \(M\) and \(N\), respectively, passing through \(A, B\). Then
    \begin{enumerate}[label=(\roman*)]
        \item each of the circles \(\gamma_M\) and \(\gamma_N\) makes equal angles with (bisects) the circles \(ABX\) and \(ABY\);
        \item \(N\) and \(M\) are respectively the external and internal homothety centers of the circles \(ABX\) and \(ABY\).
    \end{enumerate}
\end{prop}

\begin{proof}
    Let \(I\) be the incenter of \(ABC\), and \(S, T\) be the circumcenters of \(ABX, ABY\), respectively; see Figure~\ref{fig:iso-conj-bisect-circ}. Recall that \(I\) lies on \(\gamma_M\).

    Since \(AI\) bisects \(\angle XAY\) and \(BI\) bisects \(\angle XBY\), easy angle chasing yields that \(AM\) bisects \(\angle TAS\). Hence \(\gamma_M\) bisects the circles \(ABX\) and \(ABY\). But \(AN\perp AM\) so \(AN\) bisects \(\angle TAS\) too and \(\gamma_N\) also bisects the circles \(ABX\) and \(ABY\). Thus part (i) is proved.

    By the bisector property \(\frac{SM}{MT}=\frac{SA}{AT}=\frac{SN}{NT}\).
    This proves part (ii).
    \begin{figure}[H]
    \definecolor{ffvvqq}{rgb}{1,0.3333333333333333,0}
\definecolor{qqqqff}{rgb}{0,0,1}
\definecolor{fuqqzz}{rgb}{0.9568627450980393,0,0.6}
\definecolor{ttzzqq}{rgb}{0.2,0.6,0}
\definecolor{yqqqyq}{rgb}{0.5019607843137255,0,0.5019607843137255}

\begin{tikzpicture}[line cap=round,line join=round,x=1cm,y=1cm]

\clip(-2.5,-3.5) rectangle (8.3,3.2);

\coordinate (A) at (3.72,2.16);
\coordinate (B) at (3.72,-2.44);
\coordinate (C) at (-1.48,-1.42);
\coordinate (I) at (2.24457581726134,-0.6470491370193452);
\coordinate (M) at (4.687865610124604,-0.14);
\coordinate (N) at (-1.7456348408938342,-0.14);
\coordinate (S) at (7.410689655172414,-0.14);
\coordinate (T) at (3.213240584498305,-0.14);
\coordinate (X) at (3.14,-0.96);
\coordinate (Y) at (0.9146081703915153,-0.6529270221164108);

\draw [shift={(B)},line width=0.5pt] (0,0) -- (129.4510580374523:0.6) arc (129.4510580374523:147.50230594282235:0.6) -- cycle;
\draw [shift={(B)},line width=0.5pt] (0,0) -- (111.39981013208224:0.7) arc (111.39981013208224:129.4510580374523:0.7) -- cycle;
\draw [shift={(A)},line width=0.5pt] (0,0) -- (-102.42546187593098:0.6) arc (-102.42546187593098:-67.17809993034403:0.6) -- cycle;
\draw [shift={(A)},line width=0.5pt] (0,0) -- (-67.17809993034402:0.7) arc (-67.17809993034402:-31.930737984757062:0.7) -- cycle;
\draw[line width=0.5pt] (3.4592996392223228,2.0502943853287867) -- (3.569005253893536,1.789594024551109) -- (3.8297056146712136,1.8992996392223227) -- (A) -- cycle; 
\draw [line width=0.5pt,color=yqqqyq] (1.4711153846153846,-0.14) circle (3.2167502255092186cm);

\draw [line width=0.5pt,color=ttzzqq] (C)-- (A);
\draw [line width=0.5pt,color=ttzzqq] (A)-- (B);
\draw [line width=0.5pt,color=ttzzqq] (B)-- (C);

\draw [line width=0.5pt,color=ffvvqq] (S) circle (4.3486998207276475cm);
\draw [line width=0.5pt,color=ffvvqq] (T) circle (2.355165621607028cm);
\draw [line width=0.5pt,color=qqqqff] (M) circle (2.4953484404511266cm);
\draw [line width=0.5pt,color=qqqqff] (N) circle (5.929853641869465cm);

\draw [line width=0.5pt,dash pattern=on 1pt off 1pt] (N)-- (S);
\draw [line width=0.5pt] (Y)-- (B);
\draw [line width=0.5pt] (B)-- (X);
\draw [line width=0.5pt] (I)-- (B);
\draw [line width=0.5pt] (T)-- (A);
\draw [line width=0.5pt] (A)-- (M);
\draw [line width=0.5pt] (A)-- (S);
\draw [line width=0.5pt] (A)-- (N);

\draw [shift={(A)},line width=0.5pt] (-102.42546187593098:0.6) arc (-102.42546187593098:-67.17809993034403:0.6);
\draw [shift={(A)},line width=0.5pt] (-102.42546187593098:0.49) arc (-102.42546187593098:-67.17809993034403:0.49);
\draw [shift={(A)},line width=0.5pt] (-67.17809993034402:0.7) arc (-67.17809993034402:-31.930737984757062:0.7);
\draw [shift={(A)},line width=0.5pt] (-67.17809993034402:0.59) arc (-67.17809993034402:-31.930737984757062:0.59);

\begin{scriptsize}
\draw [fill=black] (A) circle (1.5pt);
\draw[color=black] (3.76,2.5) node {$A$};
\draw [fill=black] (C) circle (1.5pt);
\draw[color=black] (-1.74,-1.49) node {$C$};
\draw [fill=black] (B) circle (1.5pt);
\draw[color=black] (3.72,-2.69) node {$B$};
\draw [fill=fuqqzz] (X) circle (1.5pt);
\draw[color=fuqqzz] (3.33,-0.76) node {$X$};
\draw [fill=fuqqzz] (Y) circle (1.5pt);
\draw[color=fuqqzz] (1.1,-0.45) node {$Y$};
\draw [fill=qqqqff] (N) circle (1.5pt);
\draw[color=qqqqff] (-2.04,-0.03) node {$N$};
\draw [fill=qqqqff] (M) circle (1.5pt);
\draw[color=qqqqff] (4.88,0.1) node {$M$};
\draw [fill=black] (I) circle (1.5pt);
\draw[color=black] (2.45,-0.53) node {$I$};
\draw [fill=ffvvqq] (T) circle (1.5pt);
\draw[color=ffvvqq] (3.46,0.1) node {$T$};
\draw [fill=ffvvqq] (S) circle (1.5pt);
\draw[color=ffvvqq] (7.72,0.01) node {$S$};
\draw[color=qqqqff] (5.4,2.48) node {$\gamma_M$};
\draw[color=qqqqff] (3.6,3) node {$\gamma_N$};
\end{scriptsize}

\end{tikzpicture}
        \caption{}
        \label{fig:iso-conj-bisect-circ}
    \end{figure}
\end{proof}

\begin{prop}
\label{prop:hom-centr-bisect}
    Let us be given two spheres \(\Omega_1\) and \(\Omega_2\) intersecting each other through the circle \(\sigma\). Let \(S\) be the center of their external homothety. Then the sphere \(\Gamma\) centered at \(S\) and passing through \(\sigma\) bisects the spheres \(\Omega_1\) and \(\Omega_2\).
\end{prop}
\begin{proof}
    Consider a section of the construction by any plane \(\pi\) passing through the centers of \(\Omega_1\) and \(\Omega_2\); see Figure~\ref{fig:hom-centr-bisect}. Let \(\omega_i=\Omega_i\cap \pi, i\in\{1, 2\}\) and \(\gamma=\Gamma\cap\pi\). Let \(A\) be one of the points of intersection of \(\omega_1\) and \(\omega_2\). Let \(SA\) intersect \(\omega_1\) second time at \(B\).

    In force of homothety the tangents of \(\omega_1\) at \(B\) and of \(\omega_2\) at \(A\) are parallel. On the other hand, \(SA\) makes equal angles with the tangents of \(\omega_1\) at \(A\) and \(B\). Thus \(SA\) bisects the angle between the tangents of \(\omega_1\) and \(\omega_2\) at \(A\). Equivalently, this angle is bisected by the tangent at \(A\) to \(\gamma\), as needed.
    \begin{figure}[H]
    \definecolor{uuuuuu}{rgb}{0.26666666666666666,0.26666666666666666,0.26666666666666666}

\begin{tikzpicture}[line cap=round,line join=round,x=1cm,y=1cm]

\clip(-14.4,-1.6) rectangle (-4,2.6);

\coordinate (A) at (-8.44,-0.38);
\coordinate (B) at (-10.372680320803584,-0.05837255663053087);
\coordinate (S) at (-13.968340108392901,0.54);

\draw [shift={(B)},line width=0.5pt,fill=black, opacity=0.1] (0,0) -- (-61.2940474565994:0.4507888805409467) arc (-61.2940474565994:-9.448304829549565:0.4507888805409467) -- cycle;
\draw [shift={(A)},line width=0.5pt,fill=black, opacity=0.1] (0,0) -- (170.55169517045044:0.4507888805409467) arc (170.55169517045044:222.39743779749992:0.4507888805409467) -- cycle;
\draw [shift={(A)},line width=0.5pt,fill=black, opacity=0.1] (0,0) -- (-61.294047456598946:0.4507888805409467) arc (-61.294047456598946:-9.448304829549558:0.4507888805409467) -- cycle;
\draw [line width=0.5pt] (-9.28,0.54) circle (1.245792920191795cm);
\draw [line width=0.5pt] (-6.76,0.54) circle (1.9154111830100602cm);
\draw [line width=0.5pt] (S)-- (-5.816395970272575,2.786924415970038);
\draw [line width=0.5pt] (S)-- (-5.900994087531032,-1.6836065948475607);
\draw [line width=0.5pt] (S) circle (5.604368327837282cm);
\draw [line width=0.5pt] (-9.611034004384978,-1.449204960525404)-- (-10.704088955779593,0.5468084289778358);
\draw [line width=0.5pt] (-9.611034004384978,-1.449204960525404)-- (-7.98,0.04);
\draw [line width=0.5pt] (S)-- (-6.906744160747517,-0.6351571257294352);
\draw [line width=0.5pt] (-8.796736056524558,0.2714310597405002)-- (-7.902202895625565,-1.362064277553316);
\begin{scriptsize}
\draw [fill=uuuuuu] (A) circle (1pt);
\draw[color=uuuuuu] (-8.13432006010519,-0.26) node {$A$};
\draw[color=black] (-9.88,1.36832456799398) node {$\omega_1$};
\draw[color=black] (-6.1,2.15) node {$\omega_2$};
\draw [fill=uuuuuu] (-8.44,1.46) circle (1pt);

\draw [fill=uuuuuu] (S) circle (1pt);
\draw[color=uuuuuu] (-14.16,0.57) node {$S$};
\draw [fill=uuuuuu] (B) circle (1pt);
\draw[color=uuuuuu] (-10.222975206611578,0.2) node {$B$};
\draw [fill=uuuuuu] (-9.611034004384978,-1.449204960525404) circle (1pt);
\draw[color=black] (-8.8,2.3) node {$\gamma$};
\end{scriptsize}
\end{tikzpicture}
        \caption{}
        \label{fig:hom-centr-bisect}
    \end{figure}
\end{proof}

\section{Isogonal conjugates symmetric in bimedians}
\label{sect:iso-sym-bimed}

From here on we will denote by \(ABCD\) our isosceles tetrahedron, by \(\Omega\) its circumsphere and by \(O\) its center.

In this section we will heavily rely on coordinate bashing (though incorporating it with some crucial geometric reasoning) and will use several quantitative characteristics of the tetrahedron in terms of its coordinates. Namely, embed \(ABCD\) into the Cartesian coordinate system \(Oxyz\) so that
\begin{equation}
\label{eq:embed-tetr}
    A=(-a, b, c),\ B=(a, -b, c),\ C=(a, b, -c),\ D=(-a, -b, -c)
\end{equation}
for some \(a, b, c\in \mathbb R\setminus\{0\}\).

The following proposition defines pretty much all the values that we need:
\begin{prop}
\label{prop:quant-in-iso-tetr}
    In \(ABCD\)
    \begin{enumerate}[label=(\roman*)]
        \item if \(S\) is the area of a face then
        \[S=2\sqrt{a^2b^2+b^2c^2+c^2a^2};\]
        \item if \(d\) is the distance from \(O\) to a face then
        \[d=\frac{2|abc|}{S};\]
        \item if \(\theta\) is half the angle of dihedron \(CD\) (or equivalently \(AB\)) then
        \[\sin\theta=\frac{d}{|c|}.\]
    \end{enumerate}
\end{prop}
\begin{proof}
    \begin{enumerate}[label=(\roman*)]
        \item Taking into account that the sides of \(ABC\) are equal to \(AB=2\sqrt{a^2+b^2}\) etc. and utilizing Heron's formula, after some manipulations we find the presented formula for \(S\).
        \item Note that the volume of \(ABCD\) is equal to \([ABCD]=\displaystyle\frac{4d S}{3}\). On the other hand \([ABCD]\) is third the volume of the circumscribed parallelepiped which is rectangular in case of isosceles tetrahedron, i.e. \([ABCD]=\displaystyle\frac{8|abc|}{3}\). Hence we find the presented value for \(d\).
        \item Let \(M\) be the midpoint of \(CD\) and \(Q\) be the circumcenter of \(ACD\). Then \(\theta =\angle OMQ\) so
        \[\sin\theta = \frac{OQ}{OM}=\frac{d}{|c|}.\]
    \end{enumerate}
\end{proof}

Recall that a \emph{bimedian} of tetrahedron is a line joining midpoints of opposite edges. We will denote by \(\ell_A, \ell_B, \ell_C\) the bimedians of \(ABCD\) joining the midpoints of edges \(DA\) and \(BC\), \(DB\) and \(AC\), \(DC\) and \(AB\), respectively.

We will need the following fact too to prove the upcoming theorem:
\begin{lemma}
\label{lemma:equidist-proj}
    Let \(X\) be any point of the space. Let \(P\) and \(Q\) be its projections on \(ACD\) and \(BCD\), respectively. Also let \(R\) be its projection on \(\ell_C\). Then \(PR=RQ\).
\end{lemma}
\begin{proof}
    Let \(S\) be the projection of \(X\) on the bisector plane of dihedron \(CD\) and let \(T\) be the projection of \(S\) on \(CD\); see Figure~\ref{fig:equidist-lemma}. Then \(X, P, S, Q, T\) lie on a circle with diameter \(XT\). Since \(TS\) bisects \(\angle PTQ\) we get that \(PS=SQ\). Hence in force of \(RS\perp PQS\) we deduce that \(PR=RQ\).
    \begin{figure}[H]
    \definecolor{ttffcc}{rgb}{0,0.8,0.8}
\definecolor{ffwwzz}{rgb}{1,0.4,0.6}
\definecolor{qqccqq}{rgb}{0,0.8,0}
\definecolor{ffwwqq}{rgb}{1,0.4,0}
\definecolor{cczzqq}{rgb}{0.8,0.6,0}
\definecolor{yqqqyq}{rgb}{0.5019607843137255,0,0.5019607843137255}
\definecolor{ubqqys}{rgb}{0.29411764705882354,0,0.5098039215686274}
\definecolor{qqwwzz}{rgb}{0,0.4,0.6}
\definecolor{uuuuuu}{rgb}{0.26666666666666666,0.26666666666666666,0.26666666666666666}

\begin{tikzpicture}[line cap=round,line join=round,x=1cm,y=1cm]

\coordinate (A) at (-4.986790856836006,4.998131876940993);
\coordinate (B) at (4.986790856836006,4.889578902419431);
\coordinate (C) at (-1.7449770216271667,-5.644066017421905);
\coordinate (D) at (1.7449770216271667,-4.243644761938518);
\coordinate (X) at (-5.255973508900273,1.7440645971142985);
\coordinate (P) at (-3.937738601055789,2.3383265821303536);
\coordinate (T) at (-0.5998707507247708,-5.184566797178974);
\coordinate (Q) at (0.9942290965444196,-1.7692122614121968);
\coordinate (R) at (0,1.7256527237678778);
\coordinate (S) at (-0.5998707507247714,1.4849413162691159);

\clip(-9,-6) rectangle (9,5.2);

\draw [shift={(T)},line width=0.8pt,color=yqqqyq,fill=yqqqyq,fill opacity=0.24] (0,0) -- (90:0.6) arc (90:90+23.926607075015504:0.6) -- cycle;
\draw [shift={(T)},line width=0.8pt,color=yqqqyq,fill=yqqqyq,fill opacity=0.24] (0,0) -- (90-25.020581586895116:0.7) arc (90-25.020581586895116:90:0.7) -- cycle;

\tkzMarkRightAngle[fill=ttffcc, fill opacity=0.3,draw=ttffcc,size=.35](X,P,T)
\tkzMarkRightAngle[fill=ttffcc, fill opacity=0.3,draw=ttffcc,size=.35](X,Q,T)
\tkzMarkRightAngle[fill=ttffcc, fill opacity=0.3,draw=ttffcc,size=.3](P,S,R)
\tkzMarkRightAngle[fill=ttffcc, fill opacity=0.3,draw=ttffcc,size=.3](Q,S,R)

\draw [line width=1.2pt,color=qqwwzz] (A)-- (B);
\draw [line width=1.2pt,color=qqwwzz] (B)-- (C);
\draw [line width=1.2pt,color=qqwwzz] (C)-- (A);
\draw [line width=1.2pt,color=qqwwzz] (C)-- (D);
\draw [line width=1.2pt,color=qqwwzz] (D)-- (B);
\draw [line width=1.2pt,dash pattern=on 6pt off 6pt,color=qqwwzz] (A)-- (D);

\draw [line width=1.2pt,dash pattern=on 3pt off 3pt,color=ubqqys] (0,5.7150968304703245)-- (0,-5.497567193324396);

\draw [rotate around={-68.13578440701231:(-2.927922129812525,-1.7202511000323384)},line width=1.2pt,dash pattern=on 2pt off 2pt,color=yqqqyq] (-2.927922129812525,-1.7202511000323384) ellipse (4.188806208590344cm and 3.881354276637269cm);

\draw [line width=1.2pt,color=cczzqq] (P)-- (X);
\draw [line width=1.2pt,color=cczzqq] (X)-- (Q);
\draw [line width=1.2pt,color=cczzqq] (S)-- (R);

\draw [line width=1.2pt,color=ffwwqq] (R)-- (P);
\draw [line width=1.2pt,color=qqccqq] (P)-- (S);
\draw [line width=1.2pt,color=qqccqq] (S)-- (Q);
\draw [line width=1.2pt,color=ffwwqq] (Q)-- (R);
\draw [line width=1.2pt,color=ffwwzz,  dash pattern=on 8pt off 8pt] (S)-- (T);
\draw [line width=1.2pt,color=ffwwzz] (T)-- (Q);
\draw [line width=1.2pt,color=ffwwzz, dash pattern=on 8pt off 8pt] (P)-- (T);

\draw[color=black] (-5.5,1.9) node {$X$};
\draw[color=black] (-5.3,5) node {$A$};
\draw[color=black] (5.3,4.9) node {$B$};
\draw[color=black] (-2.05,-5.53) node {$C$};
\draw[color=black] (2.05,-4.24) node {$D$};
\draw[color=black] (-3.9,2.67) node {$P$};
\draw[color=black] (1.3,-1.76) node {$Q$};
\draw[color=black] (0.3,1.76) node {$R$};
\draw[color=black] (-0.87,1.2) node {$S$};
\draw[color=black] (-0.6,-5.5) node {$T$};
\draw[color=black] (0.3,4) node {$\ell_C$};

\fill[fill=qqwwzz, opacity=0.05] (A)--(B)--(C);
\fill[fill=qqwwzz, opacity=0.05] (D)--(B)--(C);

\begin{scriptsize}
\draw [fill=uuuuuu] (A) circle (2pt);
\draw [fill=uuuuuu] (B) circle (2pt);
\draw [fill=uuuuuu] (C) circle (2pt);
\draw [fill=uuuuuu] (D) circle (2pt);
\draw [fill=uuuuuu] (X) circle (2pt);
\draw [fill=uuuuuu] (P) circle (2pt);
\draw [fill=uuuuuu] (Q) circle (2pt);
\draw [fill=uuuuuu] (S) circle (2pt);
\draw [fill=uuuuuu] (T) circle (2pt);
\draw [fill=uuuuuu] (R) circle (2pt);
\end{scriptsize}

\end{tikzpicture}
        \caption{}
        \label{fig:equidist-lemma}
    \end{figure}
\end{proof}

\begin{theorem}
\label{theorem:bimed-hyperb-parab}
    Let \(\ell\) be a bimedian of \(ABCD\). Then the pairs of isogonal conjugate points in \(ABCD\) which are symmetric with respect to \(\ell\) form a hyperbolic paraboloid.
\end{theorem}
\begin{proof}
    Without loss of generality we may assume that \(\ell=\ell_C\). Let \(P\) and \(Q\) be isogonal conjugate points symmetric in \(\ell\). Then their midpoint \(M\) lies on \(\ell\). According to Corollary~\ref{cor:iso-conj-pedal-sphere} the pedal spheres of \(P\) and \(Q\) coincide and have the center \(M\). Denote this sphere by \(\Gamma\).
    
    Since \(M\) lies on the bisector of dihedron \(CD\) the circles \(\gamma_A\) and \(\gamma_B\) cut from \(\Gamma\) by the planes \(BCD\) and \(ACD\), respectively, are symmetric in the bisector of dihedron \(CD\).

    The projections of \(P\) and \(Q\) on \(BCD\) lie on \(\gamma_A\) so \(P\) and \(Q\) lie on the straight cylinder \(\mathfrak C_A\) based on \(\gamma_A\). Similarly \(P\) and \(Q\) lie on the straight cylinder \(\mathfrak C_B\) based on \(\gamma_B\). On the other hand \(P\) and \(Q\) lie on a plane \(\pi\) perpendicular to \(\ell\). Thus \(P\) and \(Q\) lie on the ellipses \(\pi\cap \mathfrak C_A\) and \(\pi\cap \mathfrak C_B\). However these ellipses coincide since they are symmetric in \(M\), both of them have the center \(M\) and minor axes parallel to \(CD\). Name this ellipse \(\varepsilon_1\).

    Repeating the reasoning above for the dihedron \(AB\) we find out that \(P\) and \(Q\) lie on the ellipse \(\varepsilon_2\) defined similarly. Hence \(P\) and \(Q\) are an opposite pair of the points of intersection \(\varepsilon_1\cap \varepsilon_2\).

    Now recall the embedding~\eqref{eq:embed-tetr}. Then \(M=(0, 0, z_0)\) for some \(z_0\in\mathbb R\). Note that \(\ell\) coincides with \(Oz\) and the other two bimedians of \(ABCD\) coincide with \(Ox\) and \(Oy\).



    Let \(d\) be the distance from \(O\) to faces of \(ABCD\). Then it is easy to see that the distances \(d_1\) and \(d_2\) from \(M\) to \(ACD\) and \(ABC\), respectively, are \(\displaystyle\left|\frac{z_0+c}{c}\right|d\) and \(\displaystyle\left|\frac{z_0-c}{c}\right|d\).

    If \(r_1\) and \(r_2\) are the radii of \(\gamma_B\) and \(\gamma_C=\Gamma\cap ABD\), respectively, then
    \[r_1=\sqrt{r^2-d_1^2},\quad r_2=\sqrt{r^2-d_2^2}.\]

    Clearly the minor axis of \(\varepsilon_1\) is equal to \(r_1\), while its major axis is \(\displaystyle \frac{r_1}{\sin\theta}\) where \(\theta\) is half the angle of dihedron \(CD\) (or equivalently \(AB\)).

    Note that the minor axis of \(\varepsilon_1\) is parallel to \(CD\). Thus \(\varepsilon_1\) can be obtained from the ellipse
    \[\begin{cases}
        \displaystyle\frac{x^2}{r_1^2/\sin^2\theta} + \frac{y^2}{r_1^2} = 1\\
        z=z_0
    \end{cases}\]
    by rotating it around \(Oz\) in the negative direction (from \(y\) to \(x\)) in the angle \(\varphi = \arctan\left(\frac{a}{b}\right)\). This gives the following equations for \(\varepsilon_1\):
    \[\begin{cases}
        (x\cos\varphi - y\sin\varphi)^2\sin^2\theta +(x\sin\varphi+y\cos\varphi)^2 = r_1^2\\
        z=z_0
    \end{cases}.\]
    Similarly \(\varepsilon_2\) is given by
    \[\begin{cases}
        (x\cos\varphi + y\sin\varphi)^2\sin^2\theta +(-x\sin\varphi+y\cos\varphi)^2 = r_2^2\\
        z=z_0
    \end{cases}.\]
    Subtracting these equations we get the following equations which hold for the points of \(\varepsilon_1\cap \varepsilon_2\):
    \begin{equation}
    \label{eq:intersect-ellipses}
        \begin{cases}
            -4xy\sin\varphi\cos\varphi\cos^2\theta = r_2^2-r_1^2 \\
            z=z_0
        \end{cases}.
    \end{equation}
    We have
    \[\sin\varphi=\frac{a}{\sqrt{a^2+b^2}},\quad
    \cos\varphi=\frac{b}{\sqrt{a^2+b^2}}\]
    \begin{align*}
        r_2^2-r_1^2 &= d_1^2-d_2^2 \\
        &= \frac{(z_0+c)^2-(z_0-c)^2}{c^2}d^2 = \frac{4z_0}{c}d^2.
    \end{align*}
    Using these, as well as our calculated values in Proposition~\ref{prop:quant-in-iso-tetr} we get
    \[\cos^2\theta = 1-\frac{d^2}{c^2}=\frac{c^2(a^2+b^2)}{a^2b^2+b^2c^2+c^2a^2}\]
    and~\eqref{eq:intersect-ellipses} simplifies to
    \[\begin{cases}
        \displaystyle-xy=\frac{ab}{c}z_0 \\
        z = z_0
    \end{cases}.\]
    Now letting \(z_0\) vary we get the hyperbolic paraboloid \(\mathcal H\) containing all the isogonal conjugate pairs \(P\) and \(Q\) symmetric in \(\ell\):
    \[z=-\frac{c}{ab}xy.\]

    However we still need to show that any pair of points \(P'\) and \(Q'\) on \(\mathcal H\) symmetric in \(\ell\) are isogonal conjugates. To this end we will use the equivalence of isogonal conjugate points from Corollary~\ref{cor:equiv-iso-conj-constr}.

    Let \(P_A,\ P_B,\ P_C,\ P_D\) be the projections of \(P'\) on \(BCD\), \(CDA\), \(DAB\), \(ABC\), respectively. Similarly define \(Q_A,\ Q_B,\ Q_C,\ Q_D\). Also let \(M'\) be the projection of \(P'\) (or equivalently of \(Q'\)) on \(\ell\).
    
    According to Lemma~\ref{lemma:equidist-proj} \(P_AM'=M'P_B\). In force of symmetry in \(\ell_C\) we get that all the four points \(P_A, P_B, Q_A, Q_B\) are equidistant from \(M'\). Similarly \(P_C, P_D, Q_C, Q_D\) are equidistant from \(M'\) too. So if we prove that \(P_BM'=P_CM'\) then the pedal spheres of \(P'\) and \(Q'\) will coincide and we will be done by Corollary~\ref{cor:equiv-iso-conj-constr}.

    We have the following equations for the faceplanes:
    \[ACD: \frac{x}{a}-\frac{y}{b}+\frac{z}{c}+1=0, \qquad
    ABD: \frac{x}{a}+\frac{y}{b}-\frac{z}{c}+1=0.\]
    If \(P'=(x_0, y_0, z_0)\) then
    \[M'P_B^2 = \left(x_0 - \frac{p+q}{ar}\right)^2 +
    \left(y_0 + \frac{p+q}{br}\right)^2 +
    \left(\frac{p+q}{cr}\right)^2,\]
    \[M'P_C^2 = \left(x_0 - \frac{-p+q}{ar}\right)^2 +
    \left(y_0 - \frac{-p+q}{br}\right)^2 +
    \left(\frac{-p+q}{cr}\right)^2,\]
    where
    \[p=-\frac{y_0}{b}+\frac{z_0}{c},\qquad
    q=\frac{x_0}{a}+1,\qquad
    r=\frac{1}{a^2}+\frac{1}{b^2}+\frac{1}{c^2}.\]
    Then
    \[\frac{M'P_B^2-M'P_C^2}{4} = -\frac{p}{a}\left(x_0-\frac{q}{ar}\right) +
    \frac{q}{b}\left(y_0+\frac{p}{br}\right) +
    \frac{p}{c}\cdot \frac{q}{cr}\]
    \[=-\frac{x_0}{a}p+\frac{y_0}{b}q+pq\]
    \[= -\frac{x_0}{a}\left(-\frac{y_0}{b}+\frac{z_0}{c}\right) +
    \frac{y_0}{b}\left(\frac{x_0}{a}+1\right)
    + \left(-\frac{y_0}{b}+\frac{z_0}{c}\right)\left(\frac{x_0}{a}+1\right)\]
    \[=\frac{x_0y_0}{ab}+\frac{z_0}{c}.\]
    Since \(P'=(x_0, y_0, z_0)\) lies on \(\mathcal H\) the last expression is zero, as needed.
\end{proof}

\begin{rem}
    Note that the other two bimedians different from \(\ell\) lie on \(\mathcal H\) (\(z=0\) yields \(x=0\) or \(y=0\)). Also note that the vertices of \(ABCD\) lie on \(\mathcal H\) as well.
\end{rem}

Denote by \(\mathcal H_A, \mathcal H_B, \mathcal H_C\) the hyperbolic paraboloids defined above corresponding to bimedians \(\ell_A, \ell_B, \ell_C\).

\begin{prop}
\label{prop:hparab-sphere-touch}
    For any point \(P\in \Omega\) the circles \(ABP\) and \(CDP\) touch iff \(P\in \mathcal H_C\).
\end{prop}
\begin{proof}
    Recall the embedding~\eqref{eq:embed-tetr}. Let \(P=(x, y, z)\).

    We need to check if the line \(ABP\cap CDP\) touches \(\Omega\), i.e. is perpendicular to \(OP\). Or, equivalently, whether \((x, y, z)\) lies in the linear span of the normal vectors of \(ABP\) and \(CDP\). It is not hard to check that these normals are
    \[\begin{bmatrix}
        b(z-c) \\[0.2cm]
        a(z-c) \\[0.2cm]
        -(bx+ay)
    \end{bmatrix} \quad \text{and} \quad
    \begin{bmatrix}
        b(z+c) \\[0.2cm]
        -a(z+c) \\[0.2cm]
        -bx+ay
    \end{bmatrix}.\]
    Their linear span coincides with the linear span of their half-sum and half-difference:
    \[\begin{bmatrix}
        bz \\[0.2cm]
        -ac \\[0.2cm]
        -bx
    \end{bmatrix}, \quad
    \begin{bmatrix}
        bc \\[0.2cm]
        -az \\[0.2cm]
        ay
    \end{bmatrix}.\]
    We need to check the singularity of matrix
    \[M=\begin{pmatrix}
        bz & bc & x\\[0.2cm]
        -ac & -az & y\\[0.2cm]
        -bx & ay & z
    \end{pmatrix}.\]
    Taking into account the fact that \(P\in \Omega\), i.e. \(x^2+y^2+z^2=a^2+b^2+c^2\) we get
    \begin{align*}
        \det(M) &= -abz(x^2+y^2+z^2-c^2)-xyc(a^2+b^2) \\
        &=-(a^2+b^2)(abz+xyc)
    \end{align*}
    which is zero iff \(abz+xyc=0\), that is when \(P\in \mathcal H_C\).
\end{proof}
\begin{rem}
    Obviously, similar results hold for \(\mathcal H_A\) and \(\mathcal H_B\) too.
\end{rem}

\section{Isogonal conjugation on the circumsphere}
\label{sect:iso-conj-sphere}

In this section too \(ABCD\) is an isosceles tetrahedron. All the notations are preserved.
\begin{theorem}
\label{theorem:main}
    Let \(X\notin\{A, B, C, D\}\) be a point on \(\Omega\). Then the isogonal conjugate of \(X\) with respect to \(ABCD\) lies on \(\Omega\).
\end{theorem}

\begin{proof}
    Let \(Y\) be the second intersection point of \(\Omega\) and the line joining \(D\) with the isogonal conjugate of \(X\). Points \(X\) and \(Y\) lie in isogonal planes with respect to each of the dihedrons \(DA, DB, DC\). We need to prove that they lie in isogonal planes in the dihedrons \(AB, BC, CA\) as well.

    Invert with center \(D\); see Figure~\ref{fig:inverted}. For any point \(Z\) of space denote its image by \(Z_1\). 
    
    Clearly \(A_1, B_1, C_1, X_1, Y_1\) are coplanar (they lie in the image plane of \(\Omega\)). By property~\ref{i} of isosceles tetrahedron we get
    \begin{equation*}
    \label{angle-eqs}
        \angle B_1A_1D= \angle ABD=
    \angle ACD = \angle C_1A_1D
    \tag{*}
    \end{equation*}
    and two other similar equalities.
    Taking into account this, as well as the fact that \(X_1\) and \(Y_1\) lie in isogonal planes in dihedrons
    \(DA_1, DB_1, DC_1\), we deduce that \(X_1\) and \(Y_1\) are isogonal conjugates in \(A_1B_1C_1\).

    Recall that by property~\ref{ii} \(O\) is also the incenter of \(ABCD\). This means that the plane \(eO\) is the bisector of the dihedron \(e\) for any edge \(e\) of \(ABCD\).

    It is easy to see that \(O_1\) is the reflection of \(D\) in \(A_1B_1C_1\) (consider the diametrically opposite point to \(D\)).
    Angles between spheres are preserved under inversion (Proposition~\ref{prop:angles-preserved}) so since the plane \(ABO\)
    makes equal angles with \(ABD\) and \(ABC\), the sphere \(A_1B_1O_1D\) makes equal
    angles with the plane \(A_1B_1D\) and the sphere \(A_1B_1C_1D\).
    According to Proposition~\ref{prop:bisector-sphere} this is possible only in the case when the circumcenter \(N\) of \(A_1B_1O_1D\) lies
    on the sphere \(A_1B_1C_1D\).
    Tetrahedron \(A_1B_1O_1D\) is symmetric with respect to the
    plane \(A_1B_1C_1\) so \(N\) lies on the circle \(A_1B_1C_1\).

    Note that \(N\) is the midpoint of arc \(A_1C_1B_1\). Indeed, \(NA_1=NB_1\) so \(N\) is the midpoint of either one of the two arcs \(A_1B_1\). By property~\ref{iii} the angles in~\eqref{angle-eqs} are acute. This means that the projection of \(D\) on \(A_1B_1C_1\) lies on the triangle \(A_1B_1C_1\) so the dihedrons \(A_1B_1, B_1C_1\) and \(C_1A_1\) in \(DA_1B_1C_1\) are acute. Thus \(N\) and \(C_1\) lie in the same side of \(A_1B_1\) and \(N\) is on the arc \(A_1B_1C_1\).

    \begin{figure}[H]
        \includegraphics[width=\textwidth]{figure-sources/inverted.tex}
        \caption{}
        \label{fig:inverted}
    \end{figure}
    
    According to Proposition~\ref{prop:iso-conj-bisect-circ} \(N\) is the external homothety center of circles \(A_1B_1X_1\) and \(A_1B_1Y_1\). On the other hand, \(NA_1=NB_1=NC_1\) so the line passing through the circumcenters
    of \(A_1B_1DX_1\) and \(A_1B_1DY_1\) passes through \(N\) too.
    This and the previous judgement lead to the conclusion that \(N\) is the external
    homothety center of the spheres \(A_1B_1DX_1\) and \(A_1B_1DY_1\) too.
    
    In force of Proposition~\ref{prop:hom-centr-bisect} the sphere \(A_1B_1O_1D\) with center \(N\) passing through \(A_1\) makes
    equal angles with the spheres \(A_1B_1DX_1\) and \(A_1B_1DY_1\). Therefore its preimage
    plane \(ABO\) also makes equal angles with the planes \(ABX\) and \(ABY\).
    
    Similarly \(X\) and \(Y\) are in isogonal planes with respect to the dihedrons \(BC\) and \(AC\) too, whence the conclusion follows.
\end{proof}

In his proof of Theorem~\ref{theorem:main} Ilya Bogdanov, professor at Moscow Institute of Physics and Technology, found the following interesting construction  of isogonal conjugate points on the circumsphere of isosceles tetrahedron (I. Bogdanov, personal communication, January 17, 2020):

\begin{prop}
    Let \(X\notin\{A, B, C, D\}\) be a point on \(\Omega\).
    Let \(X_A\), \(X_B\), \(X_C\) be the second intersection points of circles \(XDA\) and \(XBC\), \(XDB\) and \(XCA\), \(XDC\) and \(XAB\), respectively. Then the reflections of \(X_A\) in \(\ell_A\), \(X_B\) in \(\ell_B\) and \(X_C\) in \(\ell_C\) coincide with the isogonal conjugate of \(X\).
\end{prop}
\begin{proof}
    Let \(Y\) be the isogonal conjugate of \(X\). Let \(X'\) and \(X_A'\) be the reflections of \(X\) and \(X_A\), respectively, in \(\ell_A\). We will prove that \(Y=X_A'\). Similarly the reflections of \(X\) in \(\ell_B\) and \(\ell_C\) will coincide with \(Y\) too.
    
    Note that the planes \(DAXX_A\) and \(DAX'X_A'\), as well as \(BCXX_A\) and \(BCX'X_A'\) are pairs of isogonals. This means that \(Y\) lies on \(X'X_A'\). According to Theorem~\ref{theorem:main} \(Y\) also lies on \(\Omega\). Hence \(Y\) should coincide either with \(X'\) or \(X_A'\).

    If \(Y=X_A'\) then there is nothing to prove. Else \(Y=X'\). This means that \(Y\) and \(X\) lie on \(\mathcal H_A\). According to Proposition~\ref{prop:hparab-sphere-touch} \(X=X_A\) and \(Y=X'=X_A'\) as desired.
\end{proof}
\begin{rem}
    To prove this result Bogdanov used inversion with respect to one of the points \(\ell_A\cap \Omega\). Points \(A, B, C, D\) were mapped to the vertices of a parallelogram and angle chasing finished the proof. However we chose a different approach based on the results already proven.
\end{rem} 

\printbibliography

\end{document}